\documentclass[leqno,11pt]{amsart}\usepackage{psfrag}
\usepackage[dvips]{color}
\usepackage{multicol}
\usepackage{picinpar}
\usepackage{enumerate}
\usepackage[hyphens]{url}        
\usepackage[breaklinks,colorlinks=true,linkcolor=blue,citecolor=blue, urlcolor=blue]{hyperref}                                                 
\usepackage[utf8x]{inputenc}
\usepackage{txfonts}
\usepackage{graphicx}
\usepackage[T1]{fontenc}

\usepackage{amsmath,amstext,amssymb,amsopn,amsthm,mathrsfs}
\theoremstyle{plain}
\newtheorem{teo}{Theorem}[section]
\newtheorem{defi}{Definition}[section]
\newtheorem{cor}{Corollary}[section]
\newtheorem{lem}{Lemma}[section]
\newtheorem{prop}{Proposition}[section]
\newtheorem{obs}{Observation}[section]
\numberwithin{equation}{section}

\begin{document}

\title[Variable Gaussian Besov-Lipschitz \&   Triebel-Lizorkin spaces ] {Some results on variable Gaussian Besov-Lipschitz and variable  Gaussian Triebel-Lizorkin spaces.}

\author{Ebner Pineda}
\address{Escuela Superior Polit\'ecnica del Litoral. ESPOL, FCNM, Campus Gustavo Galindo Km. 30.5 V\'ia Perimetral, P.O. Box 09-01-5863, Guayaquil, ECUADOR.}
\email{epineda@espol.edu.ec}
\author{Luz Rodriguez}
\address{Escuela Superior Polit\'ecnica del Litoral. ESPOL, FCNM, Campus Gustavo Galindo Km. 30.5 V\'ia Perimetral, P.O. Box 09-01-5863, Guayaquil, ECUADOR.}
\email{luzeurod@espol.edu.ec}
\author{Wilfredo~O.~Urbina}
\address{Department of Mathematics, Actuarial Sciences and Economics, Roosevelt University, Chicago, IL,
   60605, USA.}
\email{wurbinaromero@roosevelt.edu}

\subjclass[2010] {Primary 42B25, 42B35; Secondary 46E30, 47G10}

\keywords{Hermite expansions, variable exponent, Besov-Lipschitz,  Triebel-Lizorkin ,  Gaussian measure.}

\begin{abstract}
In a previous paper \cite{Pinurb} two of the authors introduced and study  Gaussian  Besov-Lipschitz spaces  $B_{p,q}^{\alpha}(\gamma_{d})$ and Gaussian  Triebel-Lizorkin spaces $F_{p,q}^{\alpha}(\gamma_{d})$. Now,
in this paper we introduce the variable  Gaussian  Besov-Lipschitz spaces $B_{p(\cdot),q(\cdot)}^{\alpha}(\gamma_{d})$ and  the variable Gaussian  Triebel-Lizorkin spaces $F_{p(\cdot),q(\cdot)}^{\alpha}(\gamma_{d}),$ that is to say,  Gaussian  Besov-Lipschitz and Triebel-Lizorkin spaces with variable exponents, following  \cite{st1} and \cite{Pinurb} under certain additional regularity conditions on the exponents $p(\cdot)$ and $q(\cdot)$ introduced by Dalmasso and Scotto in  \cite{st1}. Trivially, they include the Gaussian  Besov-Lipschitz spaces  $B_{p,q}^{\alpha}(\gamma_{d})$ and Gaussian\\  Triebel-Lizorkin spaces $F_{p,q}^{\alpha}(\gamma_{d})$. We consider some inclusion relations of those spaces and finally we also prove some interpolation results for them.\end{abstract}

\maketitle

\section{Introduction and Preliminaries}
 Let us consider the Gaussian measure
\begin{equation}
\gamma_d(dx)=\frac{e^{-\|x\|^2}%
}{\pi^{d/2}} dx, \, x\in\mathbb{R}^d
\end{equation}
on  $\mathbb{R}^d$  and the Ornstein-Uhlenbeck
differential operator
\begin{equation}\label{OUop}
L=\frac12\triangle_x-\left\langle x,\nabla _x\right\rangle.
\end{equation}

Let $\nu=(\nu _1,...,\nu_d)$ be a
multi-index such that $\nu _i \geq 0, i= 1, \cdots, d$,  let $\nu
!=\prod_{i=1}^d\nu _i!,$ $\left| \nu \right| =\sum_{i=1}^d\nu _i,$ $%
\partial _i=\frac \partial {\partial x_i},$ for each $1\leq i\leq d$ and $%
\partial ^\nu =\partial _1^{\nu _1}...\partial _d^{\nu _d}$.\\
Consider the normalized Hermite polynomials of order $\nu$ in $d$ variables,
\begin{equation}
h_\nu (x)=\frac 1{\left( 2^{\left| \nu \right| }\nu
!\right)
^{1/2}}\prod_{i=1}^d(-1)^{\nu _i}e^{x_i^2}\partial _i^{\nu _i}(e^{-x_i^2}).
\end{equation}
 The {\em Ornstein-Uhlenbeck semigroup} on ${\mathbb{R}
^d}$ \label{semOU2}
 is defined by
\begin{eqnarray*}
 T_t f(x) &=&  \frac{1}{(1-e^{-2t})^{d/2}}\int_{\mathbb{R}^d}
e^{- \frac{e^{-2t}(|x|^2 +|y|^2) - 2 e^{-t} \langle x,y \rangle
}{1-e^{-2t}}} f(y) \gamma_d(dy) \quad .
\end{eqnarray*}
Using
the {\em Bochner subordination formula}
\begin{equation} \label{bochner}
 e^{-\lambda} = \frac{1}{\sqrt \pi} \int_0^{\infty} \frac{e^{-u}}{\sqrt u} e^{-\lambda^2/4u} du,
\end{equation}
we introduce the {\em
Poisson-Hermite semigroup} \label{semPH1} by
\begin{eqnarray}
P_t f(x) & = & \frac{1}{\sqrt \pi} \int_0^{\infty}
\frac{e^{-u}}{\sqrt u} T_{t^2/4u}f(x) du.
\end{eqnarray}
 Now, taking the change of
variables $s=\displaystyle\frac{t^2}{4u}$ , $ P_{t}f(x)$ can be written as
\begin{equation}\label{02}
P_t f(x) = \int^{\infty}_0 T_sf(x)  \mu^{(1/2)}_t(ds),
\end{equation}
where $$\mu^{(1/2)}_t(ds) = \frac{t}{2\sqrt{\pi}} e^{-t^2/4s} s^{-3/2} ds,$$\label{onesided1/2}
is {\em the  one-sided stable measure on $(0, \infty)$ of order $1/2$}.\\
It is easy to see that $\mu^{(1/2)}_t$ is a  probability measure on $(0, \infty)$.

It is well known, that the Hermite polynomials are
eigenfunctions of the operator $L$,
\begin{equation}\label{eigen}
L h_{\nu}(x)=-\left|\nu \right|h_\nu(x).
\end{equation}
In consequence
\begin{equation}\label{OUHerm}
T_t h_\nu(x)=e^{-t\left| \nu\right|}h_\nu(x),
\end{equation}
and
\begin{equation}\label{PHHerm}
 P_t h_\nu(x)=e^{-t\sqrt{\left| \nu\right|}}h_\nu(x),
\end{equation}
i.e. the Hermite polynomials are also eigenfunctions of $T_t$ and $P_t$ for any $t \geq 0$, for more details, see \cite{urbina2019}. \\

Next, we present some technical results for the measure $\mu^{(1/2)}_t$  needed  in what follows. First, as $\mu^{(1/2)}_t(ds) = \displaystyle\frac t{2\sqrt{\pi}}\frac{e^{-t^2/4s}}{s^{3/2}}ds = g(t,s) ds$, for any $k\in \mathbb{N}$,
we use the notation $\frac{\partial^{k}}{\partial t^{k}}\mu_{t}^{(1/2)}(ds)$ for
 \begin{equation}
\frac{\partial^{k}}{\partial t^{k}}\mu_{t}^{(1/2)}(ds) := \frac{\partial^{k} g(t,s)}{\partial t^{k}}ds.
\end{equation}

\begin{lem} \label{lematec1} Given $k\in \mathbb{N}$
\begin{equation}
 \frac{\partial^{k}\mu_{t}^{(1/2)}}{\partial
t^{k}}(ds)=\left(\sum_{i,j}a_{i,j}\frac{t^{i}}{s^{j}}\right)\mu_{t}^{(1/2)}(ds)
\end{equation}
where $\{a_{i,j}\}$ is a finite set of constants and the indexes $i\in
\mathbb{Z}$, $j\in \mathbb{N}$ verifies the ecuation $2j-i=k$.
\end{lem}
\begin{lem}\label{lematec2}  Given $k\in \mathbb{N}$ and $t>0$
$$
\int_{0}^{+\infty}\frac{1}{s^{k}}\mu_{t}^{(1/2)}(ds)=\frac{C_{k}}{t^{2k}},$$
where
$C_{k}=\frac{2^{2k}\Gamma(k+\frac{1}{2})}{\pi^{\frac{1}{2}}}.$
\end{lem}
\begin{cor}\label{corol1} Given $k\in \mathbb{N}$ and $t>0$
\begin{equation}
 \int_{0}^{+\infty}\left|\frac{\partial^{k}\mu_{t}^{(1/2)}}{\partial
t^{k}}\right|(ds)\leq\frac{C_{k}}{t^{k}}.
\end{equation}
\end{cor}
On the other hand, by considering the {\em maximal function of the Ornstein-Uhlenbeck  semigroup}
$$T^{\ast}f(x)=\displaystyle\sup_{t>0}|T_{t}f(x)|,$$
we obtain the inequaility:

\begin{lem}\label{lematec5}
Let $f\in L^{1}(\gamma_{d}), x\in \mathbb{R}^{d}$ and $k\in \mathbb{N}$
\begin{equation}
\left|\frac{\partial^{k}P_{t}f(x)}{\partial t^{k}}\right| \leq C_{k} \,
T^{\ast}f(x) t^{-k} \quad, \forall t>0 \quad .
\end{equation}
\end{lem}
For the proofs of these technical results, see \cite{Pinurb} or \cite{urbina2019}.\\

Now, for completeness, we need some background on variable Lebesgue spaces with respect to a Borel measure $\mu$.\\

A $\mu$-measurable function $p(\cdot):\Omega \subset \mathbb{R}^d \rightarrow [1,\infty]$ is said to be an {\em exponent function}; the set of all the exponent functions will be denoted by  $\mathcal{P}(\Omega,\mu)$. For $E\subset\Omega$ we set $$p_{-}(E)=\text{ess}\inf_{x\in E}p(x) \;\text{and}\; p_{+}(E)=\text{ess}\sup_{x\in E}p(x).$$
$\Omega_{\infty}=\{x\in \Omega:p(x)=\infty\}$.\\
  We use the abbreviations $p_{+}=p_{+}(\Omega)$ and $p_{-}=p_{-}(\Omega)$.

\begin{defi}\label{deflogholder}
Let $E\subset \mathbb{R}^{d}$. we say that $p(\cdot):E\rightarrow\mathbb{R}$ is locally log-H\"{o}lder continuous, denote  by $p(\cdot)\in LH_{0}(E)$, if there exists a constant $C_{1}>0$ such that
			\begin{eqnarray*}
				|p(x)-p(y)|&\leq&\frac{C_{1}}{log(e+\frac{1}{|x-y|})}
			\end{eqnarray*}
			for all $x,y\in E$.
			
			We say that $p(\cdot)$ is log-H\"{o}lder continuous at infinity with base point at $x_{0}\in \mathbb{R}^{d}$, and denote this by $p(\cdot)\in LH_{\infty}(E)$, if there exist  constants $p_{\infty}\in\mathbb{R}$ and $C_{2}>0$ such that
			\begin{eqnarray*}
				|p(x)-p_{\infty}|&\leq&\frac{C_{2}}{log(e+|x-x_{0}|)}
			\end{eqnarray*}
			for all $x\in E$.
			
			We say that $p(\cdot)$ is log-H\"{o}lder continuous, and denote this by $p(\cdot)\in LH(E)$ if both conditions are satisfied.
			The maximum, $\max\{C_{1},C_{2}\}$ is called the log-H\"{o}lder constant of $p(\cdot)$.
\end{defi}

\begin{defi}\label{defPdlog}
			Let $E\subset \mathbb{R}^{d}$, we say that  $p(\cdot)\in\mathcal{P}_{d}^{log}(E)$, if $\frac{1}{p(\cdot)}$ is log-H\"{o}lder continuous and  denote by $C_{log}(p)$ or $C_{log}$ the log-H\"{o}lder constant of $\frac{1}{p(\cdot)}$.
		\end{defi}
		
		\begin{defi}
Let $\Omega\subset \mathbb{R}^{d}$ and  $p(\cdot)\in\mathcal{P}(\Omega,\mu)$. For a $\mu$-measurable function $f:\Omega\rightarrow \overline{\mathbb{R}}$, we define the modular \begin{equation}
\rho_{p(\cdot),\mu}(f)=\displaystyle\int_{\Omega\setminus\Omega_{\infty}}|f(x)|^{p(x)}\mu(dx)+\|f\|_{L^{\infty}(\Omega_{\infty},\mu)},
\end{equation}
and the norm
\begin{equation}
\|f\|_{L^{p(\cdot)}(\Omega,\mu)}=\inf\left\{\lambda>0:\rho_{p(\cdot),\mu}(f/\lambda)\leq 1\right\}.
\end{equation}

\end{defi}

\begin{defi} The variable exponent Lebesgue space on $\Omega\subset\mathbb{R}^{d}$, $L^{p(\cdot)}(\Omega,\mu)$ consists on those $\mu\_$measurable functions $f$ for which there exists $\lambda>0$ such that $\rho_{p(\cdot),\mu}\left(\frac{f}{\lambda}\right)<\infty,$ i.e.
\begin{equation*}
L^{p(\cdot)}(\Omega,\mu) =\left\{f:\Omega\to \overline{\mathbb{R}}: f \; \text{is measurable and} \; \rho_{p(\cdot),\mu}\left(\frac{f}{\lambda}\right)<\infty, \; \text{for some} \;\lambda>0\right\}.
\end{equation*}

\end{defi}
		
\begin{obs}\label{obs4.2}
 When $\mu$ is the Lebesgue measure, we write  $\rho_{p(\cdot)}$ and $\|f\|_{p(\cdot)}$ instead of $\rho_{p(\cdot),\mu}$ and $\|f\|_{p(\cdot),\mu}$.
\end{obs}

\begin{teo} (Norm conjugate formula)
Let  $\nu$ a complete, $\sigma$-finite measure on $\Omega$.  $p(\cdot)\in \mathcal{P}(\Omega,\nu)$, then\\
\begin{equation}\label{normaconjugada}
\frac{1}{2}\|f\|_{p(\cdot),\nu}\leq \|f\|^{'}_{p(\cdot),\nu}\leq 2\|f\|_{p(\cdot),\nu},
\end{equation}
for all $f$ $\nu$-measurable on $\Omega$,\\ donde $\displaystyle\|f\|^{'}_{p(\cdot),\nu}=\sup\left\{\int_{\Omega}|f||g|d\mu:g\in L^{p'(\cdot)}(\Omega,\nu),\|g\|_{p'(\cdot),\nu}\leq 1\right\}.$
\end{teo}
\begin{proof} See Corollary 3.2.14 in \cite{LibroDenHarjHas}
\end{proof}
\begin{teo}
(H\"older's inequality) Let $\nu$ a complete, $\sigma$-finite measure on $\Omega$. $r(\cdot),q(\cdot)\in \mathcal{P}(\Omega,\nu)$, define $p(\cdot)\in \mathcal{P}(\Omega,\nu)$ by $\displaystyle \frac{1}{p(x)}=\frac{1}{q(x)}+\frac{1}{r(x)}$ $\nu$ a.e. $x\in \Omega$ .\\
Then for all $f\in L^{q(\cdot)}(\Omega,\nu)$ and $g\in L^{r(\cdot)}(\Omega,\nu)$, $ fg\in L^{p(\cdot)}(\Omega,\nu)$ and
\begin{equation}\label{Holder generalizada}
\|fg\|_{p(\cdot),\nu}\leq 2\|f\|_{q(\cdot),\nu}\|g\|_{r(\cdot),\nu}
\end{equation}
\end{teo}
\begin{proof} See Lemma 3.2.20 in \cite{LibroDenHarjHas}
\end{proof}
\begin{teo}
(Minkowski's integral inequality for variable Lebesgue spaces) Given $\mu$ and $\nu$ complete $\sigma$-finite measures on $X$ and $Y$ respectively, $p\in \mathcal{P}(X,\mu)$. Let $f:X\times Y\rightarrow \overline{\mathbb{R}}$ measurable with respect to the product measure on $X\times Y$, such that for almost every $y\in Y$, $f(\cdot,y)\in L^{p(\cdot)}(X,\mu)$. Then
\begin{equation}\label{integralMinkowski}
\left\|\int_{Y}f(\cdot,y)d\nu(y)\right\|_{p(\cdot),\mu}\leq 4\int_{Y}\|f(\cdot,y)\|_{p(\cdot),\mu}d\nu(y)
\end{equation}
\end{teo}
\begin{proof} It is completely analogous to the proof of Corollary 2.38 in \cite{dcruz} by interchanging the Lebesgue measure for complete $\sigma$-finite measures $\mu$ and $\nu$ on $X$ and $Y$ respectively, and by using (\ref{Holder generalizada}), Fubini's theorem and then (\ref{normaconjugada}).
\end{proof}

In what follows $\mu$ represents the measure $\displaystyle\mu(dt)=\frac{dt}{t}$ on $\mathbb{R}^{+}$.\\
\begin{obs}
For a $\mu$-measurable function $f:\mathbb{R}^{+}\rightarrow \overline{\mathbb{R}}$, $q(\cdot)\in\mathcal{P}(\mathbb{R}^{+},\mu)$, and any $\lambda>0$
\begin{eqnarray*}
\rho_{q(\cdot),\mu}(\frac{f}{\lambda})&=&\displaystyle\int_{0}^{\infty}\left|\frac{f(t)}{\lambda}\right|^{q(t)}\mu(dt)=\displaystyle\int_{0}^{\infty}\left|\frac{t^{-1/q(t)}f(t)}{\lambda}\right|^{q(t)}dt\\
&=&\rho_{q(\cdot)}\left(\frac{t^{-1/q(\cdot)}f}{\lambda}\right)
\end{eqnarray*}
Thus, \begin{equation}\label{normaq,dt/t}
\|f\|_{q(\cdot),\mu}=\|t^{-1/q(\cdot)}f\|_{q(\cdot)}
\end{equation}
\end{obs}
Next, we present an useful technical results for the measure $\mu$.
\begin{lem}\label{desigualdadesnormamu}
For $q(\cdot)\in\mathcal{P}(\mathbb{R}^{+},\mu)$
\begin{enumerate}
\item[i)]  For any $\alpha,\beta>0$ and $q_{+}<\infty,\; \|t^{\alpha}e^{-t\beta}\|_{q(\cdot),\mu}<\infty.$
\item[ii)]   For any $  \alpha>0, \; \|t^{\alpha}\chi_{(0,1]}\|_{q(\cdot),\mu}<\infty.$
\item[iii)] For any $ \alpha>0, \; \|t^{-\alpha}\chi_{(1,\infty)}\|_{q(\cdot),\mu}<\infty.$
\item[iv)] For any $ t_{0}>0, \; (\ln 2)^{\frac{1}{q_{-}}}\leq\|\chi_{[t_{0}/2,t_{0}]}\|_{q(\cdot),\mu}\leq 1.$
\end{enumerate}
\end{lem}
\begin{proof} Let us prove $i)$. Set $f=t^{\alpha}e^{-t\beta}$\\
\begin{eqnarray*}
\rho_{q(\cdot),\mu}(f)&=&\int_{0}^{\infty}|f(t)|^{q(t)}\mu(dt)=\int_{0}^{1}|t^{\alpha}e^{-t\beta}|^{q(t)}\frac{dt}{t}+\int_{1}^{\infty}|t^{\alpha}e^{-t\beta}|^{q(t)}\frac{dt}{t}
\end{eqnarray*}
Now, \begin{eqnarray*}
\int_{0}^{1}|t^{\alpha}e^{-t\beta}|^{q(t)}\frac{dt}{t}&=&\int_{0}^{1}t^{\alpha q(t)-1}e^{-t\beta q(t)}dt \leq\int_{0}^{1}t^{\alpha -1}dt<\infty,
\end{eqnarray*}
 since $ \alpha,\beta>0$ and $0\leq t\leq 1.$
On the other hand, by making the change of variables $u=t\beta q_{-}$
\begin{eqnarray*}
\int_{1}^{\infty}|t^{\alpha}e^{-t\beta}|^{q(t)}\frac{dt}{t}&=&\int_{1}^{\infty}t^{\alpha q(t)}e^{-t\beta q(t)}\frac{dt}{t}\\
&\leq&\int_{1}^{\infty}t^{\alpha q_{+}}e^{-t\beta q_{-}}\frac{dt}{t}
\leq\int_{0}^{\infty}t^{\alpha q_{+}}e^{-t\beta q_{-}}\frac{dt}{t}\\
&=&\int_{0}^{\infty}(\frac{u}{\beta q_{-}})^{\alpha q_{+}}e^{-u}\frac{du}{u}
=\frac{1}{(\beta q_{-})^{\alpha q_{+}}}\int_{0}^{\infty}u^{\alpha q_{+}-1}e^{-u}du\\
&=&\frac{1}{(\beta q_{-})^{\alpha q_{+}}}\Gamma(\alpha q_{+})<\infty .
\end{eqnarray*}
 since $ \alpha,\beta>0$ and $q_{+}<\infty$. Thus, $\rho_{q(\cdot),\mu}(f)<\infty$, and therefore
 $$\|t^{\alpha}e^{-t\beta}\|_{q(\cdot),\mu}<\infty.$$

The proof  of $ii)$ and $iii)$ are immediate. Now, in order to prove $iv)$, set $g=\chi_{[t_{0}/2,t_{0}]}$\\
\begin{eqnarray*}
\rho_{q(\cdot),\mu}(g)&=&\int_{0}^{\infty}|g(t)|^{q(t)}\mu(dt)=\int_{t_{0}/2}^{t_{0}}\frac{dt}{t}=\ln 2<1.
\end{eqnarray*}
Then,  $\lambda\geq 1$ implies  $\rho_{q(\cdot),\mu}(\frac{g}{\lambda})\leq \rho_{q(\cdot),\mu}(g)\leq 1$.
Thus, $\|g\|_{q(\cdot),\mu}\leq 1$.\\

 On the other hand,
taking $0<\lambda<1$
\begin{eqnarray*}
\rho_{q(\cdot),\mu}(\frac{g}{\lambda})&=&\int_{t_{0}/2}^{t_{0}}\lambda^{-q(t)}\frac{dt}{t}\geq \int_{t_{0}/2}^{t_{0}}\lambda^{-q_{-}}\frac{dt}{t}=\lambda^{-q_{-}}(\ln 2)\\
\end{eqnarray*}
then $\lambda<(\ln 2)^{1/q_{-}}$ implies $ \rho_{q(\cdot),\mu}(\frac{g}{\lambda})>1$.
Therefore, $\rho_{q(\cdot),\mu}(\frac{g}{\lambda})\leq 1$ implies $\lambda\geq(\ln 2)^{1/q_{-}}$ and then
 $$\|g\|_{q(\cdot),\mu}\geq (\ln 2)^{1/q_{-}}.$$
\end{proof}

In the case $\Omega=\mathbb{R}^{+}$, we denote $\mathcal{M}_{0,\infty}$ the set of all measurable functions $p(\cdot): \mathbb{R}^{+}\rightarrow \mathbb{R}^{+} $ which satisfy the following conditions:

$i)$ $0\leq p_{-}\leq p_{+} <\infty$,\\
$ii_{0})$ there exists $p(0)=\displaystyle\lim_{x\rightarrow 0}p(x)$ and $|p(x)- p(0)|\leq \frac{A}{\ln(1/x)}, 0< x\leq 1/2$\\
$ii_{\infty})$ there exists  $p(\infty)=\displaystyle\lim_{x\rightarrow \infty}p(x)$ and $|p(x)- p(\infty)|\leq \frac{A}{\ln(x)},  x>2$.\\
 we denote  $\mathcal{P}_{0,\infty}$ the subset of functions $p(\cdot)$ such that $p_{-}\geq 1$.\\

 Let $\alpha(\cdot),\beta(\cdot)\in LH(\mathbb{R}^{+})$, bounded with \begin{equation}\label{alpha,p'}
  \displaystyle\alpha(0)<\frac{1}{p'(0)}, \alpha(\infty)<\frac{1}{p'(\infty)}
  \end{equation}
  and
  \begin{equation}\label{beta,p}
  \displaystyle\beta(0)>-\frac{1}{p(0)}, \beta(\infty)>-\frac{1}{p(\infty)}
  \end{equation}
\begin{teo}
Let $p(\cdot)\in\mathcal{P}_{0,\infty}$, $\alpha(\cdot),\beta(\cdot)\in LH(\mathbb{R}^{+})$, bounded. Then the Hardy-type inequalities
 \begin{equation}\label{hardyinequalityvariable0ax}
\left\|x^{\alpha(x)-1}\int_{0}^{x}\frac{f(y)}{y^{\alpha(y)}}dy\right\|_{p(\cdot)}\leq C_{\alpha(\cdot),p(\cdot)}\|f\|_{p(\cdot)}
 \end{equation}
  \begin{equation}\label{hardyinequalityvariablexinfty}
\left\|x^{\beta(x)}\int_{x}^{\infty}\frac{f(y)}{y^{\beta(y)+1}}dy\right\|_{p(\cdot)}\leq C_{\beta(\cdot),p(\cdot)}\|f\|_{p(\cdot)}
 \end{equation}
are valid, if and only if, $\alpha(\cdot),\beta(\cdot)$ satisfy conditions (\ref{alpha,p'}) and (\ref{beta,p})
\end{teo}
\begin{proof} For the proof  see Theorem 3.1 and Remark 3.2 in \cite{dieningsamko}.
\end{proof}

As a consequence, we obtain the Hardy's inequalities associated to the exponent $q(\cdot)\in\mathcal{P}_{0,\infty}$ and the measure $\mu$.
\begin{cor}
Let $q(\cdot)\in\mathcal{P}_{0,\infty}$ and $r>0$, then
\begin{equation}\label{hardyineq0atr}\displaystyle\left\|t^{-r}\int_{0}^{t}g(y)dy\right\|_{q(\cdot),\mu}\leq C_{r,q(\cdot)}\left\|y^{-r+1}g\right\|_{q(\cdot),\mu}
\end{equation} and
\begin{equation}\label{hardyineqtainftyr}\displaystyle\left\|t^{r}\int_{t}^{\infty}g(y)dy\right\|_{q(\cdot),\mu}\leq C_{r,q(\cdot)}\left\|y^{r+1}g\right\|_{q(\cdot),\mu}
\end{equation}
\end{cor}
\begin{proof}
Let $\alpha(t)=-r+\frac{1}{q'(t)}=-r+1-\frac{1}{q(t)},$ for any $t\in \mathbb{R}^{+}$, $f(y)=y^{\alpha(y)}g(y), $ for any $y\in \mathbb{R}^{+}$ then $\alpha(\cdot)\in LH(\mathbb{R}^{+})$ and bounded,
$\alpha(0)=-r+\frac{1}{q'(0)}<\frac{1}{q'(0)}$
and $\alpha(\infty)=-r+\frac{1}{q'(\infty)}<\frac{1}{q'(\infty)}$.
 Then, using (\ref{normaq,dt/t}) and (\ref{hardyinequalityvariable0ax})
\begin{eqnarray*}
\left\|t^{-r}\int_{0}^{t}g(y)dy\right\|_{q(\cdot),\mu}&=&\left\|t^{-r-\frac{1}{q(t)}}\int_{0}^{t}g(y)dy\right\|_{q(\cdot)}=\left\|t^{\alpha(t)-1}\int_{0}^{t}g(y)dy\right\|_{q(\cdot)}\\
&\leq&C_{r,q(\cdot)}\left\|y^{\alpha(y)}g\right\|_{q(\cdot)}= C_{r,q(\cdot)}\left\|y^{-r+1-\frac{1}{q(y)}}g\right\|_{q(\cdot)}\\
&=&C_{r,q(\cdot)}\left\|y^{-r+1}g\right\|_{q(\cdot),\mu}.
\end{eqnarray*}
 On the other hand, by taking $\beta(t)=r-\frac{1}{q(t)}, \forall t\in \mathbb{R}^{+}$, $f(y)=y^{\beta(y)+1}g(y), \forall y\in \mathbb{R}^{+}$ then $\beta(\cdot)\in LH(\mathbb{R}^{+})$
and the proof of (\ref{hardyineqtainftyr}) is completely analogous.
\end{proof}

In what follows we will consider only Lebesgue variable spaces with respect to the Gaussian measure $\gamma_d,$ $L^{p(\cdot)}(\mathbb{R}^{d},\gamma_d).$  The next condition was introduced by E. Dalmasso and R. Scotto in \cite{DalSco} and it is crucial to deal with the Gaussian measure.

\begin{defi}\label{defipgamma}
Let $p(\cdot)\in\mathcal{P}(\mathbb{R}^{d},\gamma_{d})$, we say that $p(\cdot)\in\mathcal{P}_{\gamma_{d}}^{\infty}(\mathbb{R}^{d})$ if there exist constants $C_{\gamma_{d}}>0$ and $p_{\infty}\geq1$ such that
\begin{equation}
   |p(x)-p_{\infty}|\leq\frac{C_{\gamma_{d}}}{|x|^{2}},
\end{equation}
for $x\in\mathbb{R}^{d}\setminus\{(0,0,\ldots,0)\}.$
\end{defi}

\begin{obs}\label{obs4.2}
 It can be proved that if $p(\cdot)\in\mathcal{P}_{\gamma_{d}}^{\infty}(\mathbb{R}^{d})$, then $p(\cdot)\in LH_{\infty}(\mathbb{R}^{d})$.
\end{obs}

\section{The main results}
In this section we are going to define the  variable Gaussian Besov-Lipschitz spaces and the  variable Gaussian Triebel-Lizorkin spaces, which are the main goal of the paper.\\

The following two technical results are needed for defining  variable Gaussian Besov-Lipschitz spaces.
\begin{lem}\label{lematec3}
Let $p(\cdot)\in\mathcal{P}_{\gamma_{d}}^{\infty}(\mathbb{R}^{d})\cap LH_{0}(\mathbb{R}^{d})$ and $f\in L^{p(\cdot)}(\gamma_d), \alpha\geq 0$ y $k,l$ integers greater than
 $\alpha$, then
$$\left\|\frac{\partial^{k}u(\cdot,t)}{\partial
t^{k}}\right\|_{p(\cdot),\gamma_d}\leq A_{k}t^{-k+\alpha}\, \mbox{if and only if}
\, \,\left\|\frac{\partial^{l}u(\cdot,t)}{\partial
t^{l}}\right\|_{p(\cdot),\gamma_d}\leq A_{l}t^{-l+\alpha}.$$ Moreover, if
$A_{k}(f),A_{l}(f)$  are the smallest constants in the inequalities above then there exist constants
$A_{k,l,\alpha,p(\cdot)}$ and $D_{k,l,\alpha}$ such that
$$ A_{k,l,\alpha,p(\cdot)} A_{k}(f)\leq A_{l}(f)\leq
D_{k,l,\alpha}A_{k}(f),$$ for all $f\in L^{p(\cdot)}(\gamma_d)$.
\end{lem}

\begin{proof} Let us suppose without loss of generality that $k\geq l$. We start by proving
the direct implication. For this we use the
representation of the Poisson-Hermite semigroup  (\ref{02}), this is,
$$P_{t}f(x)=\int_{0}^{+\infty}T_{s}f(x)\mu_{t}^{(1/2)}(ds).$$
Then, by differentiating $k$-times with respect to $t$ and by using the dominated convergence theorem, we get

$$\frac{\partial^{k}P_{t}f(x)}{\partial
t^{k}}=\int_{0}^{+\infty}T_{s}f(x)
\frac{\partial^{k}\mu_{t}^{(1/2)}}{\partial t^{k}}(ds).$$

By using Lemma  \ref{lematec5}, it's easy to prove that for all
$m\in \mathbb{N}$
$$\lim_{t\rightarrow +\infty}\frac{\partial^{m}P_{t}f(x)}{\partial
t^{m}}=0.$$ Now, given $n\in \mathbb{N}$, $ n>\alpha$
\begin{eqnarray*}
-\int_{t}^{+\infty}\frac{\partial^{n+1}P_{s}f(x)}{\partial
s^{n+1}}ds&=&-\lim_{s\rightarrow
+\infty}\frac{\partial^{n}P_{s}f(x)}{\partial
s^{n}}+\frac{\partial^{n}P_{t}f(x)}{\partial t^{n}}\\
&=&\frac{\partial^{n}P_{t}f(x)}{\partial t^{n}}.
\end{eqnarray*}
Thus, for Minkowski's  integral inequality (\ref{integralMinkowski})
\begin{eqnarray*}
\left\|\frac{\partial^{n}u(\cdot,t)}{\partial t^{n}}\right\|_{p(\cdot),\gamma_d}
&\leq&4 \int_{t}^{+\infty}\left\|\frac{\partial^{n+1}u(\cdot,s)}{\partial
s^{n+1}}\right\|_{p(\cdot),\gamma_d}ds
\leq 4\int_{t}^{+\infty}A_{n+1}(f)s^{-(n+1)+\alpha}ds\\
&=&4\frac{A_{n+1}(f)}{n-\alpha}t^{-n+\alpha}.
\end{eqnarray*}
Therefore
$$A_{n}(f)\leq 4\frac{A_{n+1}(f)}{n-\alpha},$$ and, since
$n>\alpha$ is arbitrary, then, by using the above result
$k-l$ times, we obtain
\begin{eqnarray*}
A_{l}(f)&\leq&4\frac{A_{l+1}(f)}{l-\alpha}\leq 4^{2}\frac{A_{l+2}(f)}{(l-\alpha)(l+1-\alpha)}\\
&\leq&...\leq 4^{k-l}\frac{A_{k}(f)}{(l-\alpha)(l+1-\alpha)...(k-1-\alpha)}= D_{k,l,\alpha}A_{k}(f).
\end{eqnarray*}
To prove the converse, we use again the
representation (\ref{02}) and we obtain that
$$u(x,t_{1}+t_{2})=P_{t_{1}}(P_{t_{2}}f)(x)=\displaystyle\int_{0}^{+\infty}T_{s}(P_{t_{2}}f)(x)\mu_{t_{1}}^{\frac{1}{2}}(ds).$$
Thus, taking $t=t_{1}+t_{2}$ and differentiating $l$ times with
respect to $t_{2}$ and $k-l$ times with respect to $t_{1}$, we get
\begin{equation}\label{PoissonkDev}
\displaystyle\frac{\partial^{k}u(x,t)}{\partial
t^{k}}=\int_{0}^{+\infty}T_{s}
(\frac{\partial^{l}P_{t_{2}}f(x)}{\partial
t_{2}^{l}})\frac{\partial^{k-l}\mu_{t_{1}}^{\frac{1}{2}}}{\partial
t_{1}^{k-l}} (ds).
\end{equation}
 Then, by Corollary \ref{corol1},  Minkowski's integral inequality (\ref{integralMinkowski}) and the $L^{p(\cdot)}$-boundedness
 of the Ornstein-Uhlenbeck semigroup (see \cite{MorPinUrb}), we get
\begin{eqnarray*}
\left\|\frac{\partial^{k}u(\cdot,t)}{\partial
t^{k}}\right\|_{p(\cdot),\gamma_d}&\leq&4\int_{0}^{+\infty}\left\|T_{s}
\left(\frac{\partial^{l}P_{t_{2}}f}{\partial
t_{2}^{l}}\right)\right\|_{p(\cdot),\gamma_d}\left|\frac{\partial^{k-l}
\mu_{t_{1}}^{\frac{1}{2}}}{\partial
t_{1}^{k-l}}(ds)\right|\\
&\leq&4C_{p(\cdot)}\left\|\frac{\partial^{l}P_{t_{2}}f}{\partial
t_{2}^{l}}\right\|_{p(\cdot),\gamma_d}\int_{0}^{+\infty}\left|\frac{\partial^{k-l}
\mu_{t_{1}}^{\frac{1}{2}}}{\partial
t_{1}^{k-l}}(ds)\right|\\
&\leq&4C_{p(\cdot)}\left\|\frac{\partial^{l}}{\partial
t_{2}^{l}}P_{t_{2}}f\right\|_{p(\cdot),\gamma_d}C_{k-l}t_{1}^{l-k} \leq 4C_{p(\cdot)}
A_{l}(f)C_{k-l}t_{2}^{-l+\alpha}t_{1}^{l-k}.
\end{eqnarray*}
Therefore, taking $t_{1}=t_{2}=\frac{t}{2}$,
$$\displaystyle\left\|\frac{\partial^{k}u(\cdot,t)}{\partial
t^{k}}\right\|_{p(\cdot),\gamma_d}\leq 4C_{p(\cdot)}
A_{l}(f)C_{k-l}(\frac{t}{2})^{-k+\alpha}.$$ Thus
$$\displaystyle
A_{k}(f)\leq 4C_{p(\cdot)}\frac{C_{k-l}}{2^{-k+\alpha}} A_{l}(f).$$
\end{proof}

\begin{lem}\label{lematec4}
Let $p(\cdot)\in\mathcal{P}_{\gamma_{d}}^{\infty}(\mathbb{R}^{d})\cap LH_{0}(\mathbb{R}^{d})$ and $q(\cdot)\in\mathcal{P}_{0,\infty}$. Let $\alpha\geq 0$ and $k,l$ integers greater than $\alpha$. Then
$$\left\|t^{k-\alpha}\left\|\frac{\partial^{k}u(\cdot,t)}{\partial
t^{k}}\right\|_{p(\cdot),\gamma_d}\right\|_{q(\cdot),\mu}<\infty$$
if and only if
$$\left\|t^{l-\alpha}\left\|\frac{\partial^{l}u(\cdot,t)}{\partial
t^{l}}\right\|_{p(\cdot),\gamma_d}\right\|_{q(\cdot),\mu}<\infty.
 $$ Moreover, there exist constants $A_{k,l,\alpha,p(\cdot)}$ and $D_{k,l,\alpha,q(\cdot)}$ such that
\begin{eqnarray*}
D_{k,l,\alpha,q(\cdot)}\left\|t^{l-\alpha}\left\|\frac{\partial^{l}u(\cdot,t)}{\partial
t^{l}}\right\|_{p(\cdot),\gamma_d}\right\|_{q(\cdot),\mu}
&\leq&\left\|t^{k-\alpha}\left\|\frac{\partial^{k}u(\cdot,t)}{\partial
t^{k}}\right\|_{p(\cdot),\gamma_d}\right\|_{q(\cdot),\mu}\\
&\leq&A_{k,l,\alpha,p(\cdot)}\left\|t^{l-\alpha}\left\|\frac{\partial^{l}u(\cdot,t)}{\partial
t^{l}}\right\|_{p(\cdot),\gamma_d}\right\|_{q(\cdot),\mu},
\end{eqnarray*}
 for all $f\in L^{p(\cdot)}(\gamma_d)$.
\end{lem}
\begin{proof} Suppose without loss of generality that $k\geq l$. We prove first
the converse implication; by proceeding as in lemma
\ref{lematec3}, taking $t_{1}=t_{2}=\frac{t}{2}$, we have
\begin{eqnarray*}
\left\|\frac{\partial^{k}u(\cdot,t)}{\partial
t^{k}}\right\|_{p(\cdot),\gamma_d}&\leq&4C_{p(\cdot)}\left\|\frac{\partial^{l}P_{t_{2}}f}{\partial
t_{2}^{l}}\right\|_{p(\cdot),\gamma_d} C_{k-l}t_{1}^{l-k}, \,\\
&=&4C_{p(\cdot)}\left\|\frac{\partial^{l}P_{\frac{t}{2}}f}{\partial
(\frac{t}{2})^{l}}\right\|_{p(\cdot),\gamma_d}C_{k-l}(\frac{t}{2})^{l-k}.
\end{eqnarray*}
Thus
\begin{eqnarray*}
\left\|t^{k-\alpha}\left\|\frac{\partial^{k}u(\cdot,t)}{\partial
t^{k}}\right\|_{p(\cdot),\gamma_d}\right\|_{q(\cdot),\mu}&\leq&4C_{p(\cdot)}\frac{C_{k-l}}{2^{l-k}}
\left\|t^{l-\alpha}\left\|\frac{\partial^{l}u(\cdot,\frac{t}{2})}{\partial
(\frac{t}{2})^{l}}\right\|_{p(\cdot),\gamma_d}\right\|_{q(\cdot),\mu}\\
&=&A_{k,l,\alpha,p(\cdot)}\left\|s^{l-\alpha}\left\|\frac{\partial^{l}u(\cdot,s)}{\partial
s^{l}}\right\|_{p(\cdot),\gamma_d}\right\|_{q(\cdot),\mu}.
\end{eqnarray*}
with $\displaystyle A_{k,l,\alpha,p(\cdot)}=4C_{p(\cdot)}C_{k-l} 2^{k-\alpha}$.\\

For the direct implication, given $n\in \mathbb{N}$, $ n>\alpha$,
again, as in the above lemma
$$\left\|\frac{\partial^{n}u(\cdot,t)}{\partial
t^{n}}\right\|_{p(\cdot),\gamma_d}\leq 4\int_{t}^{+\infty}\left\|\frac{\partial^{n+1}u(\cdot,s)}{\partial
s^{n+1}}\right\|_{p(\cdot),\gamma_d}ds.$$

Therefore, by the Hardy's inequality (\ref{hardyineqtainftyr})
\begin{eqnarray*}
\left\|t^{n-\alpha}\left\|\frac{\partial^{n}u(\cdot,t)}{\partial
t^{n}}\right\|_{p(\cdot),\gamma_d}\right\|_{q(\cdot),\mu}
&\leq&4\left\|t^{n-\alpha}\int_{t}^{+\infty}\left\|\frac{\partial^{n+1}u(\cdot,s)}{\partial
s^{n+1}}\right\|_{p(\cdot),\gamma_d}\right\|_{q(\cdot),\mu}\\
&\leq&4C_{n,\alpha,q(\cdot)}\left\|s^{n+1-\alpha}\left\|\frac{\partial^{n+1}u(\cdot,s)}{\partial
s^{n+1}}\right\|_{p(\cdot),\gamma_d}\right\|_{q(\cdot),\mu}
\end{eqnarray*}
Now, since $n>\alpha$ is arbitrary, by using the previous result
$k-l$ times, we obtain
\begin{eqnarray*}
\left\|t^{l-\alpha}\left\|\frac{\partial^{l}u(\cdot,t)}{\partial
t^{l}}\right\|_{p(\cdot),\gamma_d}\right\|_{q(\cdot),\mu} &\leq&4C_{l,\alpha,q(\cdot)}\left\|t^{l+1-\alpha}\left\|\frac{\partial^{l+1}u(\cdot,t)}{\partial
t^{l+1}}\right\|_{p(\cdot),\gamma_d}\right\|_{q(\cdot),\mu}\\
&\leq&4^{2}C_{l,\alpha,q(\cdot)}C_{l+1,\alpha,q(\cdot)}\left\|t^{l+2-\alpha}\left\|\frac{\partial^{l+2}u(\cdot,t)}{\partial t^{l+2}}\right\|_{p(\cdot),\gamma_d}\right\|_{q(\cdot),\mu}\\
&\vdots&\\
&\leq&D_{k,l,\alpha,q(\cdot)}\left\|t^{k-\alpha}\left\|
\frac{\partial^{k}u(\cdot,t)}{\partial
t^{k}}\right\|_{p(\cdot),\gamma_d}\right\|_{q(\cdot),\mu}
\end{eqnarray*}
where $\displaystyle
D_{k,l,\alpha,q(\cdot)}=4^{k-l}C_{l,\alpha,q(\cdot)}\cdots C_{k-1,\alpha,q(\cdot)}$.
\end{proof}

The next technical result  will be the key to define the variable Gaussian Triebel-Lizorkin spaces.

\begin{lem}\label{TLind}
Let $p(\cdot)\in\mathcal{P}_{\gamma_{d}}^{\infty}(\mathbb{R}^{d})\cap LH_{0}(\mathbb{R}^{d})$ and $q(\cdot)\in\mathcal{P}_{0,\infty}$.
Let $\alpha \geq 0$ and $k,l$ integers greater than $\alpha$. Then
$$
\left\|\left\|t^{k-\alpha}\left|\frac{\partial^{k}}{\partial
t^{k}}P_{t}f\right|\right\|_{q(\cdot),\mu}\right\|_{p(\cdot),\gamma_d}<\infty $$
if and only if
$$\left\|\left\|t^{l-\alpha}\left|\frac{\partial^{l}}{\partial
t^{l}}P_{t}f\right|\right\|_{q(\cdot),\mu}\right\|_{p(\cdot),\gamma_d}<\infty.
$$
Moreover, there exist constants $A_{k,l,\alpha,p(\cdot)}, D_{k,l,\alpha,q(\cdot)}$ such that
\begin{eqnarray*}
D_{k,l,\alpha,q(\cdot)} \left\|\left\|t^{l-\alpha}\left|\frac{\partial^{l}}{\partial
t^{l}}P_{t}f\right|\right\|_{q(\cdot),\mu}\right\|_{p(\cdot),\gamma_d}&\leq&\left\|\left\|t^{k-\alpha}\left|
\frac{\partial^{k}}{\partial
t^{k}}P_{t}f\right|\right\|_{q(\cdot),\mu}\right\|_{p(\cdot),\gamma_d} \\
&\leq& A_{k,l,\alpha,p(\cdot)}  \left\|\left\|t^{l-\alpha}\left|\frac{\partial^{l}}{\partial
t^{l}}P_{t}f\right|\right\|_{q(\cdot),\mu}\right\|_{p(\cdot),\gamma_d},
\end{eqnarray*}
for all $f\in L^{p(\cdot)}(\gamma_d)$.
\end{lem}
\begin{proof}
Suppose without loss of generality that $k\geq l$. Let $n\in \mathbb{N}$ such that $n>\alpha$, we can prove that
$$\left|\frac{\partial^{n}}{\partial
t^{n}}P_{t}f(x)\right|\leq\int_{t}^{+\infty}\left|\frac{\partial^{n+1}}{\partial
s^{n+1}}P_{s}f(x)\right|ds \quad.$$

Then, by the Hardy's inequality (\ref{hardyineqtainftyr}),
\begin{eqnarray*}
\left\|t^{n-\alpha}\left|\frac{\partial^{n}}{\partial
t^{n}}P_{t}f(x)\right|\right\|_{q(\cdot),\mu}
&\leq&\left\|t^{n-\alpha}\int_{t}^{+\infty}\left|\frac{\partial^{n+1}}{\partial
s^{n+1}}P_{s}f(x)\right|ds\right\|_{q(\cdot),\mu}\\
&\leq&C_{n,\alpha,q(\cdot)}\left\|s^{n+1-\alpha}\left|\frac{\partial^{n+1}}{\partial
s^{n+1}}P_{s}f(x)\right|\right\|_{q(\cdot),\mu}
\end{eqnarray*}
Now, since $n>\alpha$ is arbitrary, by iterating the previous argument $k-l$ times, we obtain
\begin{eqnarray*}
\left\|t^{l-\alpha}\left|\frac{\partial^{l}}{\partial
t^{l}}P_{t}f(x)\right|\right\|_{q(\cdot),\mu}
&\leq&C_{l,\alpha,q(\cdot)}\left\|t^{l+1-\alpha}\left|\frac{\partial^{l+1}}{\partial
t^{l+1}}P_{t}f(x)\right|\right\|_{q(\cdot),\mu}\\
&\leq&C_{l,\alpha,q(\cdot)}C_{l+1,\alpha,q(\cdot)}\left\|t^{l+2-\alpha}\left|\frac{\partial^{l+2}}{\partial
t^{l+2}}P_{t}f(x)\right|\right\|_{q(\cdot),\mu}\\
&\vdots&\\
&\leq&C_{k,l,\alpha,q(\cdot)}\left\|t^{k-\alpha}\left|
\frac{\partial^{k}}{\partial
t^{k}}P_{t}f(x)\right|\right\|_{q(\cdot),\mu}
\end{eqnarray*}
where $\displaystyle C_{k,l,\alpha,q(\cdot)}=C_{l,\alpha,q(\cdot)} C_{l+1,\alpha,q(\cdot)}\cdots C_{k-1,\alpha,q(\cdot)}.$ Thus,
\begin{eqnarray*}
D_{k,l,\alpha,q(\cdot)}\left\|\left\|t^{l-\alpha}\left|\frac{\partial^{l}}{\partial
t^{l}}P_{t}f\right|\right\|_{q(\cdot),\mu}\right\|_{p(\cdot),\gamma_d}&\leq&\left\|\left\|t^{k-\alpha}\left|
\frac{\partial^{k}}{\partial
t^{k}}P_{t}f\right|\right\|_{q(\cdot),\mu}\right\|_{p(\cdot),\gamma_d},
\end{eqnarray*}
where $D_{k,l,\alpha,q(\cdot)} = 1/C_{k,l,\alpha,q(\cdot)}.$\\

The other inequality is obtain from the case $k=l+1$ by an inductive argument. Let $t_{1},t_{2}>0$ and take $t=t_{1}+t_{2}$, from (\ref{PoissonkDev}) we get
$$\frac{\partial^{k}u(x,t)}{\partial
t^{k}}=\int_{0}^{+\infty}T_{s}
\left(\frac{\partial^{l}P_{t_{2}}f(x)}{\partial
t_{2}^{l}}\right)\frac{\partial^{k-l}}{\partial
t_{1}^{k-l}} \mu_{t_{1}}^{(1/2)}(ds),$$
and since, $\displaystyle\frac{\partial }{\partial
t_{1}}\mu_{t_{1}}^{(1/2)} (ds)=\big(t_{1}^{-1}-\frac{t_{1}}{2s}\big)\mu_{t_{1}}^{(1/2)}(ds)$
we obtain
\begin{eqnarray*}
\left|\displaystyle\frac{\partial^{k}u(x,t)}{\partial t^{k}}\right|&\leq&\int_{0}^{+\infty}T_{s}
\left(\left|\frac{\partial^{l}P_{t_{2}}f(x)}{\partial
t_{2}^{l}}\right|\right) |\big(t_{1}^{-1}-\frac{t_{1}}{2s}\big)|\mu_{t_{1}}^{(1/2)}(ds)\\
&\leq&t_{1}^{-1}\int_{0}^{+\infty}T_{s}
\left(\left|\frac{\partial^{l}P_{t_{2}}f(x)}{\partial
t_{2}^{l}}\right|\right)\mu_{t_{1}}^{(1/2)}(ds)\\
&& \hspace{3.2cm}+\frac{t_{1}}{2}\int_{0}^{+\infty}T_{s}
\left(\left|\frac{\partial^{l}P_{t_{2}}f(x)}{\partial
t_{2}^{l}}\right|\right)\frac{1}{s} \mu_{t_{1}}^{(1/2)}(ds).
\end{eqnarray*}
Therefore
\begin{eqnarray*}
\left\|t_{2}^{k-\alpha}\left|\displaystyle\frac{\partial^{k}u(x,t)}{\partial
t^{k}}\right|\right\|_{q(\cdot),\mu}&\leq&\left\|t_{2}^{k-\alpha}t_{1}^{-1} \int_{0}^{+\infty}T_{s}
\left(\left|\frac{\partial^{l}P_{t_{2}}f(x)}{\partial
t_{2}^{l}}\right|\right)\mu_{t_{1}}^{(1/2)}(ds)\right\|_{q(\cdot),\mu}\\
&&\hspace{.15cm} + \left\|t_{2}^{k-\alpha}\frac{t_{1}}{2}\int_{0}^{+\infty}T_{s}
\left(\left|\frac{\partial^{l}P_{t_{2}}f(x)}{\partial
t_{2}^{l}}\right|\right)\frac{1}{s} \mu_{t_{1}}^{(1/2)}(ds)\right\|_{q(\cdot),\mu}.\\
&=&(I) + (II).
\end{eqnarray*}
 Now, by using Minkowski's integral inequality twice (\ref{integralMinkowski}) (since $T_{s}$ is an integral transformation with positive kernel) and the fact that $\mu_{t_{1}}^{(1/2)}(ds)$ is a probability measure, we get
\begin{eqnarray*}
(I)&=&\left\|\big(t_{2}^{k-\alpha}t_{1}^{-1}\int_{0}^{+\infty}T_{s}
\left(\left|\frac{\partial^{l}P_{t_{2}}f(x)}{\partial
t_{2}^{l}}\right|\right)\mu_{t_{1}}^{(1/2)}(ds)\right\|_{q(\cdot),\mu}\\
&\leq& 4\int_{0}^{+\infty}\left\|t_{2}^{k-\alpha}t_{1}^{-1}T_{s}
\left(\left|\frac{\partial^{l}P_{t_{2}}f(x)}{\partial
t_{2}^{l}}\right|\right)\right\|_{q(\cdot),\mu}\mu_{t_{1}}^{(1/2)}(ds)\\
&\leq&16\int_{0}^{+\infty}T_{s}\left(\left\|t_{2}^{k-\alpha}t_{1}^{-1}\left|\frac{\partial^{l}P_{t_{2}}f(x)}{\partial
t_{2}^{l}}\right|\right\|_{q(\cdot),\mu}\right)\mu_{t_{1}}^{(1/2)}(ds)\\
&\leq&16 T^{\ast} \left(\left\|t_{2}^{k-\alpha}t_{1}^{-1}\left|\frac{\partial^{l}P_{t_{2}}f(x)}{\partial
t_{2}^{l}}\right|\right\|_{q(\cdot),\mu}\right).
\end{eqnarray*}

For (II) we proceed  in analogous way,  and by using Lemma \ref{lematec2} we get
\begin{eqnarray*}
(II)&\leq&\frac{16}{2}T^{\ast}\left(\left\|t_{2}^{k-\alpha}t_{1}\left|\frac{\partial^{l}P_{t_{2}}f(x)}{\partial
t_{2}^{l}}\right|\right\|_{q(\cdot),\mu}\right)\int_{0}^{+\infty}\frac{1}{s}\mu_{t_{1}}^{(1/2)}(ds)\\
&=&8T^{\ast}\left(\left\|t_{2}^{k-\alpha}t_{1} \left|\frac{\partial^{l}P_{t_{2}}f(x)}{\partial
t_{2}^{l}}\right|\right\|_{q(\cdot),\mu}\right) C_{1}\frac{1}{t_{1}^{2}}.
\end{eqnarray*}
Now, since $T^{\ast}$ is defined as a supremum, we get
 \begin{eqnarray*}
(II) &\leq& 8C_{1}T^{\ast}\left(\left\|t_{2}^{k-\alpha}t_{1}^{-1}\left|\frac{\partial^{l}P_{t_{2}}f(x)}{\partial
t_{2}^{l}}\right|\right\|_{q(\cdot),\mu}\right).
\end{eqnarray*}
Then, taking $t_{1}=t_{2}=\frac{t}{2}$ and the change of variable $s=\frac{t}{2}$, we have
\begin{eqnarray*}
(I)&\leq&16T^{\ast}\left(\left\|s^{l-\alpha}\left|\frac{\partial^{l}P_{s}f(x)}{\partial
s^{l}}\right|\right\|_{q(\cdot),\mu}\right)
\end{eqnarray*}
and
\begin{eqnarray*}
(II) &\leq&8C_1 T^{\ast}\left(\left\|s^{l-\alpha}\left|\frac{\partial^{l}P_{s}f(x)}{\partial s^{l}}\right|\right\|_{q(\cdot),\mu}\right)
\end{eqnarray*}
Therefore, by the $L^{p(\cdot)}(\gamma_d)$-boundedness of $T^*$ (see \cite{MorPinUrb}),
\begin{eqnarray*}
\left\|\left\|t^{k-\alpha}\left|\displaystyle\frac{\partial^{k}u(\cdot,t)}{\partial
t^{k}}\right|\right\|_{q(\cdot),\mu}\right\|_{p(\cdot),\gamma_d}&\leq&2^{k-\alpha}16\left\|T^{\ast}\left(\left\|s^{l-\alpha}\left|\frac{\partial^{l}P_{s}f}{\partial
s^{l}}\right|\right\|_{q(\cdot),\mu}\right)\right\|_{p(\cdot),\gamma_d}\\
&+&2^{k-\alpha}8C_1\left\|T^{\ast}\left(\left\|s^{l-\alpha}\left|\frac{\partial^{l}P_{s}f}{\partial
s^{l}}\right|\right\|_{q(\cdot),\mu}\right)\right\|_{p(\cdot),\gamma_d}\\
&\leq&2^{k-\alpha}C_{p(\cdot)}(16+8C_{1})\left\|\left\|s^{l-\alpha}\left|\frac{\partial^{l}P_{s}f}{\partial
s^{l}}\right|\right\|_{q(\cdot),\mu}\right\|_{p(\cdot),\gamma_d}.
\end{eqnarray*}
\end{proof}

Next, we need the following technical result for the $L^{p(\cdot)}(\gamma_d)$-norms of the derivatives of the Poisson-Hermite semigroup:

\begin{lem}\label{kdecay}
Let $p(\cdot)\in\mathcal{P}_{\gamma_{d}}^{\infty}(\mathbb{R}^{d})\cap LH_{0}(\mathbb{R}^{d})$. Suppose that $f\in L^{p(\cdot)}(\gamma_d)$, then for any integer $k$,
$\displaystyle\left\|\frac{\partial^{k}}{\partial
t^{k}}P_t f\right\|_{p(\cdot),\gamma_{d}}\leq C_{p(\cdot)}\left\|\frac{\partial^{k}}{\partial
s^{k}}P_s f\right\|_{p(\cdot),\gamma_{d}}$,
for whatever $0<s<t<+\infty$. Moreover,
\begin{equation}\label{kdecayine}
\left\|\frac{\partial^{k}}{\partial t^{k}}P_t f\right\|_{p(\cdot),\gamma_d}\leq
\frac{C_{k,p(\cdot)}}{t^{k}} \|f\|_{p(\cdot),\gamma_d}, \quad t>0 \quad .
\end{equation}
\end{lem}
\begin{proof} First, let us consider the case $k=0$. Fixed $t_{1},t_{2}>0$,
by using the semigroup property of $\{P_{t}\}$, we get
$$
P_{t_{1}+t_{2}}f(x) =P_{t_{1}}(P_{t_{2}}f(x)) $$
Thus, by the $L^{p(\cdot)}$-boundedness of $\{P_{t}\}$ (see \cite{MorPinUrb}),

$$\|P_{t_{1}+t_{2}}f\|_{p(\cdot),\gamma_d}\leq C_{p(\cdot)}\|P_{t_{2}}f\|_{p(\cdot),\gamma_{d}}.$$

In order to prove the general case, $k >0$, using the dominated convergence theorem and differentiating the identity
 $u(x,t_{1}+t_{2})=P_{t_{1}}(u(x,t_{2}))$ $k$-times with respect to $t_{2}$
 we  obtain
$$ \frac{\partial^{k}u(x,t_{1}+t_{2})}{\partial (t_{1}+t_{2})^{k}}=P_{t_{1}}\left(\frac{\partial^{k}u(x,t_{2})}{\partial
t_{2}^{k}}\right)$$ and then we proceed as in the previous argument. In other to prove (\ref{kdecayine}) we use again the representation
(\ref{02}) of the Poisson-Hermite semigroup and differentiating $k$-times
with respect to $t$ to obtain
$$\frac{\partial^{k}}{\partial t^{k}}u(x,t)=\int_{0}^{+\infty}T_{s}f(x) \frac{\partial^{k}}{\partial
t^{k}}\mu_{t}^{(1/2)}(ds).$$ Thus, by the Minkowski's integral inequality, the $L^{p(\cdot)}$-boundedness of the Ornstein-Uhlenbeck semigroup (see \cite{MorPinUrb})  and the Corollary \ref{corol1} , for $t>0$
\begin{eqnarray*}
\left\|\frac{\partial^{k}u(\cdot,t)}{\partial
t^{k}}\right\|_{p(\cdot),\gamma_d}&\leq&4\int_{0}^{+\infty}\left\|T_{s}f
\frac{\partial^{k}\mu_{t}^{(1/2)}}{\partial
t^{k}}(ds)\right\|_{p(\cdot),\gamma_d}\\
&=&4\int_{0}^{+\infty}\|T_{s}f\|_{p(\cdot),\gamma_d}\left|\frac{\partial^{k}\mu_{t}^{(1/2)}}{\partial
t^{k}}(ds)\right|\\
&\leq&4C_{p(\cdot)}\|f\|_{p(\cdot),\gamma_d}\int_{0}^{+\infty}\left|\frac{\partial^{k}
\mu_{t}^{(1/2)}}{\partial t^{k}}(ds)\right| \leq
\frac{C_{k,p(\cdot)}}{t^{k}} \|f\|_{p(\cdot),\gamma_d}.
\end{eqnarray*}
\end{proof}

The Lipschitz spaces can be generalized of the following way (see, for example \cite{Pinurb}, \cite{st1},\cite{trie1},\cite{trie2}), using the Poisson-Hermite semigroup.\\

We are ready  to define the variable Gaussian Besov-Lipschitz spaces $B_{p(\cdot),q(\cdot)}^{\alpha}(\gamma_d)$, also called Gaussian  Besov-Lipschitz spaces with variable exponents
or variable Besov-Lipschitz spaces for expansions in Hermite polinomials.

\begin{defi}
Let $p(\cdot)\in\mathcal{P}_{\gamma_{d}}^{\infty}(\mathbb{R}^{d})\cap LH_{0}(\mathbb{R}^{d})$ and $q(\cdot)\in\mathcal{P}_{0,\infty}$.
Let $\alpha \geq 0$, $k$ the smallest integer greater than
$\alpha$. The variable Gaussian Besov-Lipschitz space
$B_{p(\cdot),q(\cdot)}^{\alpha}(\gamma_d)$ is defined as the set of functions $f \in
L^{p(\cdot)}(\gamma_d)$ such that
\begin{equation}\label{e15}
\left\|t^{k-\alpha} \left\|
\frac{\partial^{k}P_t f}{\partial t^{k}} \right\|_{p(\cdot),\gamma_d}
\right\|_{q(\cdot),\mu}  < \infty.
\end{equation}
The norm of $f \in B_{p(\cdot),q(\cdot)}^{\alpha}(\gamma_d)$ is defined as
\begin{equation}
\left\| f \right\|_{B_{p(\cdot),q(\cdot)}^{\alpha}}: =  \left\| f \right\|_{p(\cdot),
\gamma_d} +\left\|t^{k-\alpha} \left\|
\frac{\partial^{k}   P_t f}{\partial t^{k}} \right\|_{p(\cdot),\gamma_d}\right\|_{q(\cdot),\mu}
.
\end{equation}

 The variable Gaussian Besov-Lipschitz space
$B_{p(\cdot),\infty}^{\alpha}(\gamma_d)$ is defined as the set of functions $f \in
L^{p(\cdot)}(\gamma_d)$ for which there exists a constant $A$ such that
$$\left\|\frac{\partial^{k}P_t f}{\partial
t^{k}}\right\|_{p(\cdot),\gamma_d}\leq At^{-k+\alpha}$$ and then the norm of
$f \in B_{p(\cdot),\infty}^{\alpha}(\gamma_d)$ is defined as
\begin{equation}
\left\| f \right\|_{B_{p(\cdot),\infty}^{\alpha}}: =  \left\| f
\right\|_{p(\cdot),\gamma_d} +A_{k}(f),
\end{equation}
where $A_{k}(f)$ is the smallest constant $A$ in
the above inequality.
\end{defi}

Lemmas  \ref{lematec3} and \ref{lematec4} show that we could have
replaced  $k$ with any other integer $l$ greater than $\alpha$ and
the resulting norms are equivalents.\\

Now, let us study some inclusion relations between variable Gaussian
Besov-Lipschitz spaces. The next result is analogous to Proposition 10, page 153 in \cite{st1} (see also \cite{Pinurb} or Proposition 7.36 in \cite{urbina2019} ).

\begin{prop} \label{incluBesov}
Let $p(\cdot)\in\mathcal{P}_{\gamma_{d}}^{\infty}(\mathbb{R}^{d})\cap LH_{0}(\mathbb{R}^{d})$ and $q_{1}(\cdot),q_{2}(\cdot)\in\mathcal{P}_{0,\infty}$. The inclusion $B_{p(\cdot),q_1(\cdot)}^{\alpha_{1}}(\gamma_d)\subset
B_{p(\cdot),q_2(\cdot)}^{\alpha_{2}}(\gamma_d)$ holds if:
\begin{enumerate}
\item [i)]   $\alpha_{1}>\alpha_{2}>0$ ( $q_{1}(\cdot)$ y $q_{2}(\cdot)$ not need to be related), or
\item[ii)] If $\alpha_{1}=\alpha_{2}$ and
$q_{1}(t)\leq q_{2}(t)$ a.e.
\end{enumerate}
\end{prop}
\begin{proof}
To prove part $ii)$, let us take $\alpha$ the common value of $\alpha_{1}$ and $\alpha_{2}$.\\ Let  $f\in B_{p(\cdot),q_1(\cdot)}^{\alpha}$ and
set
$A=\displaystyle\left\|t^{k-\alpha}\left\|\frac{\partial^{k}P_t
f}{\partial
t^{k} }\right\|_{p(\cdot),\gamma_d}\right\|_{q_{1}(\cdot),\mu}$.\\
Fixed $t_{0}>0$
$$\left\|\chi_{[\frac{t_{0}}{2},t_{0}]}t^{k-\alpha}\left\|\frac{\partial^{k}P_t f}{\partial
t^{k} }\right\|_{p(\cdot),\gamma_d}\right\|_{q_{1}(\cdot),\mu}\leq A.$$ However, by Lemma \ref{kdecay},$$\displaystyle\left\|\frac{\partial^{k}P_{t_{0}} f}{\partial t_{0}^{k}
}\right\|_{p(\cdot),\gamma_d}\leq C_{p(\cdot)}\left\|\frac{\partial^{k}P_{t} f}{\partial t^{k}
}\right\|_{p(\cdot),\gamma_d}\quad t\in [\frac{t_{0}}{2},t_{0}].$$
  Thus, we obtain
 \begin{eqnarray*}
\left\|\frac{\partial^{k}P_{t_{0}} f}{\partial
t_{0}^{k} }\right\|_{p(\cdot),\gamma_d}\left\|\chi_{[\frac{t_{0}}{2},t_{0}]}t^{k-\alpha}\right\|_{q_{1}(\cdot),\mu}&\leq& C_{p(\cdot)}\left\|\chi_{[\frac{t_{0}}{2},t_{0}]}t^{k-\alpha}\left\|\frac{\partial^{k}P_t f}{\partial
t^{k} }\right\|_{p(\cdot),\gamma_d}\right\|_{q_{1}(\cdot),\mu}\\
&\leq& C_{p(\cdot)}A.
\end{eqnarray*}
Therefore,
 \begin{eqnarray*}
(\frac{t_{0}}{2})^{k-\alpha}\left\|\frac{\partial^{k}P_{t_{0}} f}{\partial t_{0}^{k} }\right\|_{p(\cdot),\gamma_d}\left\|\chi_{[\frac{t_{0}}{2},t_{0}]}\right\|_{q_{1}(\cdot),\mu}&\leq&\left\|\frac{\partial^{k}P_{t_{0}} f}{\partial t_{0}^{k} }\right\|_{p(\cdot),\gamma_d}\left\|\chi_{[\frac{t_{0}}{2},t_{0}]}t^{k-\alpha}\right\|_{q_{1}(\cdot),\mu}\\
&\leq& C_{p(\cdot)}A
\end{eqnarray*}
and by Lemma \ref{desigualdadesnormamu}
 \begin{eqnarray*}
 (\frac{t_{0}}{2})^{k-\alpha}\left\|\frac{\partial^{k}P_{t_{0}} f}{\partial
t_{0}^{k} }\right\|_{p(\cdot),\gamma_d}(\ln 2)^{1/q_{1}^{-}}&\leq&(\frac{t_{0}}{2})^{k-\alpha}\left\|\frac{\partial^{k}P_{t_{0}} f}{\partial
t_{0}^{k} }\right\|_{p(\cdot),\gamma_d}\left\|\chi_{[\frac{t_{0}}{2},t_{0}]}\right\|_{q_{1}(\cdot),\mu}\\
&\leq& C_{p(\cdot)}A
\end{eqnarray*}
 Then,
 $$\displaystyle\left\|\frac{\partial^{k}P_{t_{0}}
f}{\partial t_{0}^{k}}\right\|_{p(\cdot),\gamma_d}\leq\frac{ C_{p(\cdot)}2^{k-\alpha}}{(\ln 2)^{1/q_{1}^{-}}}At_{0}^{-k+\alpha},$$
 and since $t_{0}$ is arbitrary
$$\displaystyle\left\|\frac{\partial^{k}P_t f}{\partial t^{k}
}\right\|_{p(\cdot),\gamma_d}\leq C_{k,\alpha,p(\cdot)q_{1}(\cdot)}At^{-k+\alpha},$$
for all $t>0$. In other words, $f\in B_{p(\cdot),q_1(\cdot)}^{\alpha}$ implies that $f\in
B_{p(\cdot),\infty}^{\alpha}$.\\
Now, let us take $g(t)=\displaystyle t^{k-\alpha}\left\|\frac{\partial^{k}P_t
f}{\partial t^{k}
}\right\|_{p(\cdot),\gamma_d}$, then $\rho_{q_{1}(\cdot),\mu}(g)<\infty$, since, $f\in B_{p(\cdot),q_1(\cdot)}^{\alpha}$. Thus, as $q_{2}(t)\geq q_{1}(t)\;\text{ a.e.}$
\begin{eqnarray*}
\rho_{q_{2}(\cdot),\mu}(g)&=&\int_{0}^{+\infty}\left(t^{k-\alpha}\left\|\frac{\partial^{k}P_t
f}{\partial t^{k}
}\right\|_{p(\cdot),\gamma_d}\right)^{q_{2}(t)}\frac{dt}{t}\\
&=&\int_{0}^{+\infty}\left(t^{k-\alpha}\left\|\frac{\partial^{k}P_t
f}{\partial t^{k}
}\right\|_{p(\cdot),\gamma_d}\right)^{q_{2}(t)-q_{1}(t)}\left(t^{k-\alpha}\left\|\frac{\partial^{k}P_t
f}{\partial
t^{k} }\right\|_{p(\cdot),\gamma_d}\right)^{q_{1}(t)}\frac{dt}{t}\\
&\leq&(C_{k,\alpha,p(\cdot)q_{1}(\cdot)}A)^{q_{2}^{+}-q_{1}^{-}}\int_{0}^{+\infty}\left(t^{k-\alpha}\left\|\frac{\partial^{k}P_t
f}{\partial
t^{k} }\right\|_{p(\cdot),\gamma_d}\right)^{q_{1}(t)}\frac{dt}{t}\\
&=&(C_{k,\alpha,p(\cdot)q_{1}(\cdot)}A)^{q_{2}^{+}-q_{1}^{-}}\rho_{q_{1}(\cdot),\mu}(g)<+\infty.
\end{eqnarray*}
Hence, $f\in B_{p(\cdot),q_2(\cdot)}^{\alpha}$. In order to prove part $i)$, by Lemma \ref{kdecay},
we obtain
$$\left\|\frac{\partial^{k}P_t f}{\partial t^{k}}\right\|_{p(\cdot),\gamma_d}\leq
C_{k,p(\cdot)}t^{-k}, \, t>0.$$ Now, given $f\in B_{p(\cdot),q_1(\cdot)}^{\alpha_{1}}$,
again by setting
$$A=\left\|t^{k-\alpha_{1}}\left\|\frac{\partial^{k}P_t f}{\partial t^{k}}\right\|_{p(\cdot),\gamma_d}\right\|_{q_{1}(\cdot),\mu},$$
we obtain, as in part $ii)$,
$$\displaystyle\left\|\frac{\partial^{k}P_t f}{\partial t^{k}
}\right\|_{p(\cdot),\gamma_d}\leq C_{k,\alpha_{1},p(\cdot)q_{1}(\cdot)}At^{-k+\alpha_{1}},$$ for all $t>0$. Therefore,
\begin{eqnarray*}
\left\|t^{k-\alpha_{2}}\left\|\frac{\partial^{k}P_t
f}{\partial t^{k}
}\right\|_{p(\cdot),\gamma_d}\right\|_{q_{2}(\cdot),\mu}&\leq&\left\|\chi_{(0,1]}t^{k-\alpha_{2}}\left\|\frac{\partial^{k}P_t
f}{\partial t^{k}
}\right\|_{p(\cdot),\gamma_d}\right\|_{q_{2}(\cdot),\mu}\\
&& \hspace{1.5cm} + \left\|\chi_{(1,\infty)}t^{k-\alpha_{2}}\left\|\frac{\partial^{k}P_t
f}{\partial
t^{k} }\right\|_{p(\cdot),\gamma_d}\right\|_{q_{2}(\cdot),\mu}\\
&=&(I)+(II).
\end{eqnarray*}
Now, again by Lemma \ref{desigualdadesnormamu} we get,
\begin{eqnarray*}
(I)&=&\left\|\chi_{(0,1]}t^{k-\alpha_{2}}\left\|\frac{\partial^{k}P_t
f}{\partial t^{k}
}\right\|_{p(\cdot),\gamma_d}\right\|_{q_{2}(\cdot),\mu}\leq \left\|\chi_{(0,1]}t^{k-\alpha_{2}}C_{k,\alpha_{1},p(\cdot)q_{1}(\cdot)}At^{-k+\alpha_{1}}\right\|_{q_{2}(\cdot),\mu}\\
&=&C_{k,\alpha_{1},p(\cdot)q_{1}(\cdot)}A \left\|\chi_{(0,1]}t^{\alpha_{1}-\alpha_{2}}\right\|_{q_{2}(\cdot),\mu}<\infty,
\end{eqnarray*}
and also by Lemma \ref{desigualdadesnormamu},
\begin{eqnarray*}
(II)&=&\left\|\chi_{(1,\infty)}t^{k-\alpha_{2}}\left\|\frac{\partial^{k}P_t
f}{\partial
t^{k} }\right\|_{p(\cdot),\gamma_d}\right\|_{q_{2}(\cdot),\mu}\leq \left\|\chi_{(1,\infty)}t^{k-\alpha_{2}}C_{k,p(\cdot)}t^{-k}\right\|_{q_{2}(\cdot),\mu}\\
&=&C_{k,p(\cdot)}\left\|\chi_{(1,\infty)}t^{-\alpha_{2}}\right\|_{q_{2}(\cdot),\mu}<\infty.
\end{eqnarray*}
Hence,
\begin{eqnarray*}
\left\|t^{k-\alpha_{2}}\left\|\frac{\partial^{k}P_t
f}{\partial t^{k}
}\right\|_{p(\cdot),\gamma_d}\right\|_{q_{2}(\cdot),\mu}&<&+\infty,
\end{eqnarray*}
and then $f\in B_{p(\cdot),q_2(\cdot)}^{\alpha_{2}}$.
\end{proof}

Let us now define the variable Gaussian
Triebel-Lizorkin spaces $F_{p(\cdot),q(\cdot)}^{\alpha}(\gamma_{d})$, which represent another way to measure regularity of functions, proceeding as in \cite{Pinurb}, \cite{trie1} or \cite{trie2}.
\begin{defi}
Let $p(\cdot)\in\mathcal{P}_{\gamma_{d}}^{\infty}(\mathbb{R}^{d})\cap LH_{0}(\mathbb{R}^{d})$ and $q(\cdot)\in\mathcal{P}_{0,\infty}$.
Let $\alpha \geq  0$ and $k$ the smallest integer greater than
$\alpha$. The variable Gaussian Triebel-Lizorkin space
$F_{p(\cdot),q(\cdot)}^{\alpha}(\gamma_{d})$ is the set of functions $f\in
L^{p(\cdot)}(\gamma_d)$ such that
\begin{equation}\label{e16}
\left\| \left\|t^{k-\alpha}
\frac{\partial^{k}P_t f}{\partial t^{k}}\right\|_{q(\cdot),\mu} \right\|_{p(\cdot),\gamma_d}<\infty,
\end{equation}
The norm of $f \in F_{p(\cdot),q(\cdot)}^{\alpha}(\gamma_d)$ is defined as
\begin{equation}
\left\| f \right\|_{F_{p(\cdot),q(\cdot)}^{\alpha}}: =  \left\| f \right\|_{p(\cdot),
\gamma_d} +\left\| \left\|t^{k-\alpha}
\frac{\partial^{k}P_t f}{\partial t^{k}}\right\|_{q(\cdot),\mu}   \right\|_{p(\cdot),\gamma_d}.
\end{equation}
\end{defi}

By Lemma \ref{TLind}, the definition of $F_{p(\cdot),q(\cdot)}^{\alpha}$ is independent of the integer $k>\alpha$ chosen and the resulting norms are equivalents.
\begin{obs} The variable Gaussian Besov-Lipschitz and   variable Gaussian Triebel-Lizorkin spaces are, by construction, subspaces of $L^{p(\cdot)}(\gamma_d)$. Moreover, since trivially $\left\| f \right\|_{p(\cdot),\gamma_d} \leq \left\| f \right\|_{B_{p(\cdot),q(\cdot)}^{\alpha}}$ and $\left\| f \right\|_{p(\cdot),\gamma_d} \leq \left\| f \right\|_{F_{p(\cdot),q(\cdot)}^{\alpha}}$, the inclusions are continuous.
On the other hand, from (\ref{PHHerm}) it is clear that for all $t>0$ and $k\in \mathbb{N}$,
$$\displaystyle\frac{\partial^{k}}{\partial t^{k}}P_{t}h_{\beta}(x)=(-1)^k |\beta|^{k/2}e^{-t\sqrt{|\beta|}}h_{\beta}(x),$$
and again by Lemma \ref{desigualdadesnormamu},
\begin{eqnarray*}
\left\|t^{k-\alpha}\left\|\frac{\partial^{k}}{\partial t^{k}}P_{t}h_{\beta}\right\|_{p(\cdot),\gamma_d}\right\|_{q(\cdot),\mu}&=&\left\|t^{k-\alpha}\left\|(-|\beta|^{1/2})^{k} e^{-t\sqrt{|\beta|}}h_{\beta}\right\|_{p(\cdot),\gamma_d}\right\|_{q(\cdot),\mu}\\
&=& |\beta|^{k/2}\|h_{\beta}\|_{p(\cdot),\gamma_d}\left\|t^{k-\alpha} e^{-t\sqrt{|\beta|}}\right\|_{q(\cdot),\mu} \\
&=&C_{k,\alpha,\beta,q(\cdot)}\|h_{\beta}\|_{p(\cdot),\gamma_d}<\infty.
\end{eqnarray*}
Thus,  $h_{\beta}\in B^{\alpha}_{p(\cdot),q(\cdot)}(\gamma_{d})$ and
$$
\|h_{\beta}\|_{B^{\alpha}_{p(\cdot),q(\cdot)}}=(1+C_{k,\alpha,\beta,q(\cdot)})\|h_{\beta}\|_{p(\cdot),\gamma_d}.$$
In a similar way, $h_{\beta}\in F^{\alpha}_{p(\cdot),q(\cdot)}(\gamma_{d})$ and
\begin{eqnarray*}
\|h_{\beta}\|_{F^{\alpha}_{p(\cdot),q(\cdot)}}&=&\|h_{\beta}\|_{p(\cdot),\gamma_d}+\left\|\left\|t^{k-\alpha}|\frac{\partial^{k}}{\partial t^{k}}P_{t}h_{\beta}|\right\|_{q(\cdot),\mu}\right\|_{p(\cdot),\gamma_d}\\
&=&\|h_{\beta}\|_{p(\cdot),\gamma_d}+ |\beta|^{k/2}\left\|t^{k-\alpha} e^{-t\sqrt{|\beta|}}\right\|_{q(\cdot),\mu}\|h_{\beta}\|_{p(\cdot),\gamma_d}\\
&=&(1+C_{k,\alpha,\beta,q(\cdot)})\|h_{\beta}\|_{p(\cdot),\gamma_d} =  \|h_{\beta}\|_{B^{\alpha}_{p(\cdot),q(\cdot)}}.
\end{eqnarray*}
Hence, the polinomials ${\mathcal P}$ is contained in $ B^{\alpha}_{p(\cdot),q(\cdot)}(\gamma_{d})$ and in $ F^{\alpha}_{p(\cdot),q(\cdot)}(\gamma_{d})$.
\end{obs}

Also, we have an inclusion result for the variable Gaussian Triebel-Lizorkin spaces, which is
analogous to Proposition \ref{incluBesov},  see also \cite{Pinurb} or Proposition 7.40 in \cite{urbina2019}.

\begin{prop}\label{incluTriebel}
Let $p(\cdot)\in\mathcal{P}_{\gamma_{d}}^{\infty}(\mathbb{R}^{d})\cap LH_{0}(\mathbb{R}^{d})$ and $q_{1}(\cdot),q_{2}(\cdot)\in\mathcal{P}_{0,\infty}$.
The inclusion $F_{p(\cdot),q_1(\cdot)}^{\alpha_{1}}(\gamma_d)\subset
F_{p(\cdot),q_2(\cdot)}^{\alpha_{2}}(\gamma_d)$ holds for $\alpha_1 >
\alpha_2>0$ and $q_1(t)> q_2(t) $ a.e.
\end{prop}
\begin{proof} Let us consider $f\in F_{p(\cdot),q_1(\cdot)}^{\alpha_{1}}$, then
\begin{eqnarray*}
\left\|t^{k-\alpha_{2}}\left|\frac{\partial^{k}P_{t}f(x)}{\partial
t^{k}}\right|\right\|_{q_{2}(\cdot),\mu}
&\leq&\left\|t^{k-\alpha_{2}}\left|\frac{\partial^{k}P_{t}f(x)}{\partial
t^{k}}\right|\chi_{(0,1]}\right\|_{q_{2}(\cdot),\mu}\\
&&\hspace{1.5cm}
+ \left\|t^{k-\alpha_{2}}\left|\frac{\partial^{k}P_{t}f(x)}{\partial
t^{k}}\right|\chi_{(1,\infty)}\right\|_{q_{2}(\cdot),\mu}\\
&=& (I) + (II).
\end{eqnarray*}
Now, since $q_{1}(t)>q_{2}(t)$ a.e., by taking  $r(t)=\displaystyle\frac{q_{1}(t)q_{2}(t)}{q_{1}(t)-q_{2}(t)}$,
 we obtain that $r(\cdot)\geq 1$ and
$\displaystyle\frac{1}{r(\cdot)}+\frac{1}{q_{1}(\cdot)}=\frac{1}{q_{2}(\cdot)}$, thus, by
H\"older's inequality (\ref{Holder generalizada}) and Lemma \ref{desigualdadesnormamu}
\begin{eqnarray*}
(I)&=&\left\|t^{\alpha_{1}-\alpha_{2}}\chi_{(0,1]}t^{k-\alpha_{1}}\left|\frac{\partial^{k}P_{t}f(x)}{\partial
t^{k}}\right|\right\|_{q_{2}(\cdot),\mu}\\
&\leq&2\left\|
t^{\alpha_{1}-\alpha_{2}}\chi_{(0,1]}\right\|_{r(\cdot),\mu}
\left\|t^{k-\alpha_{1}}\left|\frac{\partial^{k}P_{t}f(x)}{\partial
t^{k}}\right|\right\|_{q_{1}(\cdot),\mu}\\
&=&C_{\alpha_{1},\alpha_{2},q_{1}(\cdot),q_{2}(\cdot)}\left\|t^{k-\alpha_{1}}\left|\frac{\partial^{k}P_{t}f(x)}{\partial
t^{k}}\right|\right\|_{q_{1}(\cdot),\mu}.
\end{eqnarray*}
Now, for the second term $(II)$, by using Lemmas
\ref{desigualdadesnormamu} and \ref{lematec5}, we get
\begin{eqnarray*}
(II)&=&\left\|t^{k-\alpha_{2}}\left|\frac{\partial^{k}P_{t}f(x)}{\partial
t^{k}}\right|\chi_{(1,\infty)}\right\|_{q_{2}(\cdot),\mu}\leq C_{k}T^{\ast}f(x) \left\|\chi_{(1,\infty)}t^{k-\alpha_{2}}t^{-k}\right\|_{q_{2}(\cdot),\mu}\\
&=&C_{k} T^{\ast}f(x)\left\|\chi_{(1,\infty)} t^{-\alpha_{2}}\right\|_{q_{2}(\cdot),\mu}=
C_{k,\alpha_{2},q_{2}(\cdot)}T^{\ast}f(x).\end{eqnarray*} Then, by using the
$L^{p(\cdot)}(\gamma_d)$ boundedness of $T^{\ast}$ (see \cite{MorPinUrb}),
\begin{eqnarray*}
\left\|\left\|t^{k-\alpha_{2}}\left|\frac{\partial^{k}P_{t}f}{\partial
t^{k}}\right|\right\|_{q_{2}(\cdot),\mu}\right\|_{p(\cdot),\gamma_d}&\leq&
C_{\alpha_{1},\alpha_{2},q_{1}(\cdot),q_{2}(\cdot)}\left\|\left\|t^{k-\alpha_{1}}\left|\frac{\partial^{k}P_{t}f}{\partial
t^{k}}\right|\right\|_{q_{1}(\cdot),\mu}\right\|_{p(\cdot),\gamma_d}\\
&+&C_{k,\alpha_{2},q_{2}(\cdot)} \|T^{\ast}f\|_{p(\cdot),\gamma_d}\\
&\leq&C_{\alpha_{1},\alpha_{2},q_{1}(\cdot),q_{2}(\cdot)}
\left\|\left\|t^{k-\alpha_{1}}\left|\frac{\partial^{k}P_{t}f}{\partial
t^{k}}\right|\right\|_{q_{1}(\cdot),\mu}\right\|_{p(\cdot),\gamma_d}\\
&+&C_{k,\alpha_{2},p(\cdot),q_{2}(\cdot)} \|f\|_{p(\cdot),\gamma_d}<+\infty \quad .
\end{eqnarray*}
Therefore, $f\in F_{p(\cdot),q_{2}(\cdot)}^{\alpha_{2}}.$ \\
\end{proof}

\section{Interpolation results}

Finally, we are going to consider some interpolation results for the Gaussian variable Besov-Lipschitz and the variable Triebel-Lizorkin spaces.\\

We will use the following results for general variable Lebesgue spaces $L^{p(\cdot)}(X,\nu)$.
\begin{lem}\label{lema3.2.6harjhuleto}
Let $p(\cdot)\in \mathcal{P}(\Omega,\nu)$ and $s>0$ such that $sp^{-}\geq 1$. Then

$\displaystyle\||f|^{s}\|_{p(\cdot),\nu}=\|f\|^{s}_{sp(\cdot),\nu}$.
\end{lem}
\begin{proof} It is the same proof of Lemma 3.2.6 in \cite {LibroDenHarjHas}.
\end{proof}

\begin{lem}\label{holderinterpolacion}
Let $\nu$ a complete $\sigma$-finite measure on $X$. $r_{j}(\cdot)\in \mathcal{P}(X,\nu)$, $1<r^{-}_{j},r^{+}_{j}<\infty$,$j=0,1$. For all $0<\lambda<1$, if $f\in L^{r_{j}(\cdot)}(X,\nu)$, $j=0,1$ then $f\in
L^{r(\cdot)}(X,\nu)$ where
$\displaystyle\frac{1}{r(y)}=\frac{1-\lambda}{r_{0}(y)}+\frac{\lambda}{r_{1}(y)},$ a.e. $y\in X$ and \begin{equation}
\|f\|_{r(\cdot),\nu}\leq 2\|f\|_{r_{0}(\cdot),\nu}^{1-\lambda}\|f\|_{r_{1}(\cdot),\nu}^{\lambda}.
\end{equation}
\end{lem}
\begin{proof} It is a consequence of H\"older's inequality (\ref{Holder generalizada}) and Lemma \ref{lema3.2.6harjhuleto}. \end{proof}

Now, we present the interpolation result.
\begin{teo}
Let $p_{j}(\cdot)\in \mathcal{P}(\mathbb{R}^{d},\gamma_{d})$ and $q_{j}\in \mathcal{P}(\mathbb{R}^{+},\mu), j=0,1$
\begin{enumerate}
\item [i)] For $1<p^{-}_{j},q^{-}_{j}$, $p^{+}_{j},q^{+}_{j}<+\infty$ and $\alpha_{j}\geq 0$, if
$f\in B_{p_{j}(\cdot),q_{j}(\cdot)}^{\alpha_{j}}(\gamma_d)$, $j=0,1,$ then
for all $0<\theta<1, f\in B_{p(\cdot),q(\cdot)}^{\alpha}(\gamma_d)$, where
$$\alpha=\alpha_{0}(1-\theta)+\alpha_{1}\theta$$
 and
$$\displaystyle\frac{1}{p(x)}=\frac{1-\theta}{p_{0}(x)}+\frac{\theta}{p_{1}(x)},\, \text{a.e.} \; x\in \mathbb{R}^{d},$$
$$\displaystyle\frac{1}{q(t)}=\frac{1-\theta}{q_{0}(t)}+\frac{\theta}{q_{1}(t)},\, \text{a.e.} \; t\in \mathbb{R}^{+}.$$

\item[ii)]  For $1<p^{-}_{j},q^{-}_{j}$, $p^{+}_{j},q^{+}_{j}<+\infty$ and $\alpha_{j}\geq 0$, if
$f\in F_{p_{j}(\cdot),q_{j}(\cdot)}^{\alpha_{j}}(\gamma_d)$, $j=0,1,$ then for all $0<\theta<1, f\in F_{p(\cdot),q(\cdot)}^{\alpha}(\gamma_d)$, where
$$\alpha=\alpha_{0}(1-\theta)+\alpha_{1}\theta,$$
and
$$\displaystyle\frac{1}{p(x)}=\frac{1-\theta}{p_{0}(x)}+\frac{\theta}{p_{1}(x)},\, \text{a.e.}  \; x\in \mathbb{R}^{d},$$
$$\displaystyle\frac{1}{q(t)}=\frac{1-\theta}{q_{0}(t)}+\frac{\theta}{q_{1}(t)},\, \text{a.e.} \; t \in \mathbb{R}^{+}.$$
\end{enumerate}
\end{teo}
\begin{proof}

\begin{enumerate}
\item[i)] Let $k$ be any integer greater than $\alpha_{0}$ and
$\alpha_{1}$, by using Lemma \ref{holderinterpolacion}, we obtain for $\alpha =\alpha_0 (1- \theta) +\alpha_1 \theta$,
$$
\left\|t^{k-\alpha}\left\|\frac{\partial^{k}P_{t}f}{\partial
t^{k}}\right\|_{p(\cdot),\gamma_d}\right\|_{q(\cdot),\mu}\hspace{5cm}.$$
\begin{eqnarray*}
&\leq&\left\|t^{k-(\alpha_0 (1- \theta) +\alpha_1
\theta)}2\left\|\frac{\partial^{k}P_{t}f}{\partial
t^{k}}\right\|_{p_{0}(\cdot),\gamma_d}^{1-\theta}\left\|\frac{\partial^{k}P_{t}f}{\partial
t^{k}}\right\|_{p_{1}(\cdot),\gamma_d}^{\theta}\right\|_{q(\cdot),\mu} \\
&=&2\left\|t^{(1-\theta)(k-\alpha_{0})+\theta(k-\alpha_{1})}\left\|\frac{\partial^{k}P_{t}f}{\partial
t^{k}}\right\|_{p_{0}(\cdot),\gamma_d}^{1-\theta}\left\|\frac{\partial^{k}P_{t}f}{\partial
t^{k}}\right\|_{p_{1}(\cdot),\gamma_d}^{\theta}\right\|_{q(\cdot),\mu}\\
&=&2\left\|\left(t^{k-\alpha_{0}}\left\|\frac{\partial^{k}P_{t}f}{\partial
t^{k}}\right\|_{p_{0}(\cdot),\gamma_d}\right)^{1-\theta}\left(t^{k-\alpha_{1}}\left\|\frac{\partial^{k}P_{t}f}{\partial
t^{k}}\right\|_{p_{1}(\cdot),\gamma_d}\right)^{\theta }\right\|_{q(\cdot),\mu}.
\end{eqnarray*}
Thus, by H\"older's inequality (\ref{Holder generalizada}) and Lemma \ref{lema3.2.6harjhuleto},
\begin{eqnarray*}
&&\left\|t^{k-\alpha}\left\|\frac{\partial^{k}P_{t}f}{\partial
t^{k} }\right\|_{p(\cdot),\gamma_d}\right\|_{q(\cdot),\mu}\\
&\leq&4\left\|t^{k-\alpha_{0}}\left\|\frac{\partial^{k}P_{t}f}{\partial
t^{k}}\right\|_{p_{0}(\cdot),\gamma_d}\right\|^{1-\theta}_{q_{0}(\cdot),\mu}\left\|t^{k-\alpha_{1}}\left\|\frac{\partial^{k}P_{t}f}{\partial
t^{k}}\right\|_{p_{1}(\cdot),\gamma_d}\right\|^{\theta}_{q_{1}(\cdot),\mu}<+\infty ,
\end{eqnarray*}
that is, $f\in B_{p(\cdot),q(\cdot)}^{\alpha}(\gamma_d)$.\\

\item[ii)] Analogously, by H\"older's inequality (\ref{Holder generalizada}) and Lemma \ref{lema3.2.6harjhuleto}, we obtain for $\alpha = \alpha_0 (1- \theta) +\alpha_1
\theta$,
\begin{eqnarray*}
\left\|t^{k-\alpha}\left|\frac{\partial^{k}P_{t}f(x)}{\partial
t^{k} }\right|\right\|_{q(\cdot),\mu}&=&\left\|\left(t^{k-\alpha_{0}}
\left|\frac{\partial^{k}P_{t}f(x)}{\partial t^{k}
}\right|\right)^{1-\theta}\left(t^{k-\alpha_{1}}\left|\frac{\partial^{k}P_{t}f(x)}{\partial
t^{k}}\right|\right)^{\theta}\right\|_{q(\cdot),\mu}\\
&\leq&2\left\|t^{k-\alpha_{0}}
\left|\frac{\partial^{k}P_{t}f(x)}{\partial
t^{k}}\right|\right\|^{1-\theta}_{q_{0}(\cdot),\mu}\left\|t^{k-\alpha_{1}}\left|\frac{\partial^{k}P_{t}f(x)}{\partial
t^{k} }\right|\right\|^{\theta}_{q_{1}(\cdot),\mu},
\end{eqnarray*}
a.e. $x\in \mathbb{R}^{d}.$
Therefore
\begin{eqnarray*}
\left\|\left\|t^{k-\alpha}\left|\frac{\partial^{k}P_{t}f}{\partial
t^{k}}\right|\right\|_{q(\cdot),\mu}\right\|_{p(\cdot),\gamma_d}&\leq&2
\left\|\left\|t^{k-\alpha_{0}}
\left|\frac{\partial^{k}P_{t}f}{\partial
t^{k}}\right|\right\|^{1-\theta}_{q_{0}(\cdot),\mu}\left\|t^{k-\alpha_{1}}\left|\frac{\partial^{k}P_{t}f}{\partial
t^{k} }\right|\right\|^{\theta}_{q_{1}(\cdot),\mu}\right\|_{p(\cdot),\gamma_d},
\end{eqnarray*}
and again by H\"older's inequality and Lemma \ref{lema3.2.6harjhuleto},
\begin{eqnarray*}
&&\left\|\left\|t^{k-\alpha}\left|\frac{\partial^{k} P_{t}f}{\partial
t^{k}}\right|\right\|_{q(\cdot),\mu}\right\|_{p(\cdot),\gamma_{d}}\\
&\leq&4 \left\|\left\|t^{k-\alpha_{0}}
\left|\frac{\partial^{k}P_{t}f}{\partial t^{k}}\right|\right\|_{q_{0}(\cdot),\mu}\right\|^{1-\theta}_{p_{0}(\cdot),\gamma_{d}}\left\|\left\|t^{k-\alpha_{1}}\left|\frac{\partial^{k}P_{t}f}{\partial
t^{k} }\right|\right\|_{q_{1}(\cdot),\mu}\right\|^{\theta}_{p_{1}(\cdot),\gamma_{d}}<+\infty .
\end{eqnarray*}
That is, $f\in F_{p(\cdot),q(\cdot)}^{\alpha}(\gamma_d)$.
\end{enumerate}
\end{proof}

\end{document}